\documentclass[12pt]{amsart}
\usepackage{amsthm}
\usepackage{amsmath}
\usepackage{amssymb}
\usepackage{comment}
\usepackage{enumitem}
\usepackage{amsfonts}
\usepackage{mathtools}
\usepackage{hyperref}
\usepackage{dynkin-diagrams}
\usepackage[utf8]{inputenc}

\usepackage{tikz,tikz-cd}

\newcommand{\bigdot}{\boldsymbol{\cdot}}

\newcommand\precdot{\mathrel{\ooalign{$\prec$\cr
  \hidewidth\raise0.001ex\hbox{$\cdot\mkern0.6mu$}\cr}}}

\newtheorem{theorem}{Theorem}[section]
\newtheorem{def-prop}[theorem]{Definition-Proposition}
\newtheorem{proposition}[theorem]{Proposition}

\newtheorem{lemma}[theorem]{Lemma}
\newtheorem{corollary}[theorem]{Corollary}

\theoremstyle{definition}
\newtheorem{ex}[theorem]{Example}
\newtheorem{definition}[theorem]{Definition}

\theoremstyle{remark}

\title{Levi-spherical Schubert varieties}
\author{Yibo Gao}
\address{Beijing International Center for Mathematical Research, Peking University, Beijing, China}
\email{gaoyibo@bicmr.pku.edu.cn}
\author{Reuven Hodges}
\address{Department of Mathematics, University of Kansas, Lawrence, KS 66045, USA}
\email{rmhodges@ku.edu}
\author{Alexander Yong}
\address{Department of Mathematics, U.~Illinois at Urbana-Champaign, Urbana, IL 61801, USA}
\email{ayong@illinois.edu}

\date{December 5, 2023}

\begin{document}
\pagestyle{plain}
\begin{abstract}
We prove a short, root-system uniform, combinatorial classification of Levi-spherical Schubert varieties for any generalized flag variety $G/B$ of finite Lie type.
We apply this to the study of multiplicity-free decompositions of a Demazure module into irreducible representations of a Levi subgroup.
\end{abstract}
\maketitle
%\tableofcontents

\section{Introduction}

\subsection{History and motivation} In his essay \cite{Howe} on representation theory and invariant theory, R. Howe discusses the significance of multiplicity-free actions as an
organizing principle for the subject. Classical invariant theory focuses on actions of a reductive group $G$ on symmetric algebras, which is to say,
coordinate rings of vector spaces. Now one also considers $G$-actions on varieties $X$ and their coordinate rings ${\mathbb C}[X]$. 
Such an action is multiplicity-free if ${\mathbb C}[X]$ decomposes, as a $G$-module, into irreducible $G$-modules each with multiplicity one.
An important example is when $X$ is the \emph{base affine space} of a complex, semisimple algebraic group $G$ \cite{BGG75}; in this case the coordinate ring is a multiplicity-free direct sum
of the irreducible representations of $G$. Lustzig's theory of dual canonical bases \cite{L90} provides a basis for it.
In the early 2000s, understanding this basis was a motivation for S.~Fomin and A.~Zelevinsky's theory of Cluster algebras \cite{FZ02}.

The notion of multiplicity-free actions is closely connected to that of \emph{spherical varieties}. Let $G$ be a connected, complex, reductive algebraic group; we say that a variety $X$
is a $G$-variety if $X$ is equipped with an action of $G$ that is a morphism of varieties. A spherical variety is a normal $G$-variety where a Borel subgroup of $G$ has an open, and therefore dense, orbit. 
An normal, affine $G$-variety $X$ is spherical if and only if ${\mathbb C}[X]$ decomposes into irreducible $G$-modules each with multiplicity
one \cite{VK78}. If $X$ is instead a normal, projective $G$-variety then one can still recover one direction of this implication. That is, if the induced $G$-action on the homogeneous coordinate ring of $X$ is multiplicity-free, then $X$ is $G$-spherical \cite[Proposition 4.0.1]{Hodges.Lakshmibai}. 

Spherical varieties possess numerous nice properties. For instance, projective spherical varieties are Mori Dream Spaces.
Moreover, Luna-Vust theory describes all the birational models of a spherical variety in terms of colored fans; these fans generalize the notion of
fans used to classify toric varieties (which are themselves spherical varieties). N.~Perrin's excellent survey covers additional background on spherical
varieties \cite{Perrin}. 

It is an open problem to classify all spherical actions on products of flag varieties. This is solved in the case of Levi subgroups; we point to the work of 
P.~Littelmann \cite{Littel}, P.~Magyar--J.~Weyman--A.~Zelevinsky \cite{MWZ, MWZ2}, J.~Stembridge \cite{STEM,Stem:Weyl}, R.~Avdeev--A.~Petukhov \cite{AP14, AP20}.
Connecting back to the representation-theoretic perspective of \cite{Howe}, in
\cite{STEM,Stem:Weyl}, J.~Stembridge relates this classification problem to the multiplicity-freeness of restrictions of \emph{Weyl modules} \cite[Lecture~6]{Fulton.Harris}. Indeed, the homogeneous
coordinate ring of a single flag variety is a multiplicity-free sum of spaces of global sections on the variety with respect to line bundles associated to each dominant integral weight. By the Borel-Weil-Bott theorem, these  spaces
are isomorphic to the irreducible representations of $G$. This is the central object of interest in \emph{Standard Monomial Theory} \cite{LR08} and is closely related to
the coordinate ring of base affine space mentioned above. As remarked above a product of flag varieties is $G$-spherical if its homogeneous coordinate ring is multiplicity-free
as an $G$-module.

This paper solves a related problem. We classify all \emph{Levi-spherical} Schubert varieties in a single flag variety; that is, Schubert varieties that are spherical for the action of a Levi
subgroup. Here, the relevant ring is the homogeneous coordinate ring of a Schubert variety and the attendant 
representation theory is that of \emph{Demazure modules} \cite{Demazure}, which are Borel subgroup representations. Critically for this paper, they are also Levi subgroup
representations. Multiplicity-freeness in this setting refers to the decomposition of these modules into irreducible Levi subgroup representations. 
This study was initiated in \cite{Hodges.Yong1} and the authors solved the problem for the $GL_n$ case in \cite{GHY}. 
In \cite{GHY2} we conjectured an answer for all finite rank Lie types; this paper proves that conjecture. During the completion of this article, we learned that M.~Can-P.~Saha \cite{CanSaha} independently
proved the conjecture.

\subsection{Background} Throughout, let $G$ be a complex, connected, reductive algebraic group and let $B\leq G$ be a choice of Borel subgroup along
with a maximal torus $T$ contained in $B$. The \emph{Weyl group} is 
$W:= N_G(T)/T$, where $N_G(T)$ is the normalizer of $T$ in $G$. 
The orbits of the homogeneous space $G/B$ under the action of $B$ by left translations are the \emph{Schubert cells} $X_w^{\circ}, w\in W$. Their Zariski closures 
\[X_w:=\overline{X_w^{\circ}}\]
 are the \emph{Schubert varieties}. It is relevant below that these varieties are normal \cite{DeCon-Lak,
 RR}.

A \emph{parabolic subgroup} of $G$ is a closed subgroup $B\subset P\subsetneq G$ such that $G/P$ is a projective variety. Each such $P$ admits a \emph{Levi decomposition}  
\[P=L\ltimes R_u(P)\]
where $L$ is a reductive subgroup called a \emph{Levi subgroup} of $P$ and $R_u(P)$ is the unipotent radical. One parabolic subgroup is
$P_w:={\rm stab}_G(X_w)$. Any of the parabolic subgroups $B\subseteq Q\subseteq P_w$ act on $X_w$. 

Let $L_Q$ be a Levi subgroup of $Q$. A variety $X$ is \emph{$H$-spherical} for the action of a complex reductive algebraic group $H$ if it is
normal and contains an open, and therefore dense, orbit of a Borel subgroup of~$H$. Our reference for spherical varieties is \cite{Perrin}; toric varieties are examples of spherical varieties.
\begin{definition}[{\cite[Definition~1.8]{Hodges.Yong1}}]
Let $B\subseteq Q\subseteq P_w$ be a parabolic subgroup of $G$. 
$X_w\subseteq G/B$ is \emph{$L_Q$-spherical} if has a dense, open orbit of a Borel subgroup of
$L_Q$ under left-translations. 
\end{definition}

\subsection{The main result} We give a root-system uniform combinatorial criterion to decide if $X_w$ is $L_Q$-spherical. Let
$\Phi:=\Phi({\mathfrak g},T)$ be the root system  of weights for the adjoint action of $T$ on the Lie algebra ${\mathfrak g}$ of $G$. It has a decomposition $\Phi=\Phi^+\cup \Phi^-$ into positive and negative roots.
Let $\Delta\subset \Phi^+$ be the base of simple roots.
 The parabolic subgroups $Q=P_I\supset B$ are in bijection with subsets $I$ of $\Delta$; let $L_I:=L_Q$. The set of \emph{left descents} of $w$ is 
 \[{\mathcal D}_L(w)=\{\beta\in \Delta: \ell(s_{\beta}w)<\ell(w)\},\]
where $\ell(w)=\dim X_w$ is the \emph{Coxeter length} of $w$. Under the bijection, $P_w=P_{{\mathcal D}_L(w)}$, and $B\subset Q\subseteq P_w=P_{{\mathcal D}_L(w)}$ satisfy $Q=P_I$ for some 
$I\subseteq {\mathcal D}_L(w)$. 

For $I\subset \Delta$, a \emph{parabolic subgroup} $W_I\subseteq W$ is the subgroup generated by 
$S_I:=\{s_{\beta}:\beta \in I\}$. A \emph{standard Coxeter element} $c\in W_I$ is any product of the elements of 
$S_I$ listed in some order. Let $w_0(I)$ be the longest element of $W_I$.  The following definition was given in type $A$ in \cite[Definition~1.1]{GHY} and in general type in \cite[Section~4]{GHY2}:
%[{\cite[Definition~1.1]{GHY}, \cite[Section~4]{GHY2}}]
\begin{definition}
Let $w\in W$ and $I\subseteq {\mathcal D}_L(w)$ be fixed. Then $w$ is $I$-spherical if $w_0(I)w$ is a standard Coxeter element for $W_{J}$ where $J\subseteq \Delta$.
\end{definition}
We first note that if $I\subseteq {\mathcal D}_L(w)$, then the left inversion set ${\mathcal I}(w)$, defined in Section~\ref{sec:2},
 contains all the positive roots in the root subsystem generated by $I$, and thus $w=w_0(I)d$ is a length-additive expression for some $d\in W$,
  by Proposition 3.1.3 in \cite{Bjorner.Brenti}. 

\begin{theorem}
\label{conj:main}
Fix $w\in W$ and $I\subseteq {\mathcal D}_L(w)$. $X_w$ is $L_I$-spherical if and only if $w$ is $I$-spherical.
\end{theorem}

Theorem~\ref{conj:main} resolves the main conjecture of the authors' earlier work \cite[Conjecture~4.1]{GHY2}
and completes the main goal set forth in \cite{Hodges.Yong1}. In \cite{GHY}, Theorem~\ref{conj:main} was established in the case $G=GL_n$ using essentially algebraic combinatorial methods concerning \emph{Demazure characters}
(or in their type $A$ embodiment, the \emph{key polynomials}). In contrast, the geometric arguments of this paper are quite different, significantly shorter, but require more background of the reader in algebraic groups. Theorem~\ref{conj:main} is a generalization of work of P.~Karuppuchamy \cite{Karrup} that characterizes Schubert varieties that are toric, which in our setup means spherical for the action of $L_{\emptyset}=T$. Using work of 
R.~S.~Avdeev--A. V. Petukhov \cite{AP14}, Theorem~\ref{conj:main} may also be seen as a generalization of some results of P.~Magyar--J.~Weyman--A.~Zelevinsky \cite{MWZ} and J.~Stembridge~\cite{STEM, Stem:Weyl} on
spherical actions on $G/B$; see \cite[Theorem~2.4]{Hodges.Yong1}. Previously, there was not even a finite algorithm to decide $L_I$-sphericality of $X_w$ in general.

\subsection{Organization} Examples of the main result are given in Section~\ref{sec:5}. Sections~\ref{sec:2} and~\ref{sec:3} prove Theorem~\ref{conj:main}. 
Section~\ref{sec:4} offers an application of our main result to the study of Demazure modules \cite{Demazure}.

\section{Examples of Theorem~\ref{conj:main}}\label{sec:5}
We begin with a few examples illustrating Theorem~\ref{conj:main}.

\begin{ex}[$E_8$ {cf.~\cite[Example~1.3]{Hodges.Yong1}}]
The $E_8$ Dynkin diagram is
\dynkin[label]E8. 
One associates the simple roots $\beta_i$ ($1\leq i\leq 8$) with 
this labeling and writes $s_i:=s_{\beta_i}$. Suppose 
\[w=s_2 s_3 s_4 s_2 s_3 s_4 s_5 s_4 s_2 s_3 s_1 s_4 s_5 s_6 s_7 s_6 s_8 s_7 s_6 \in W.\]
Then ${\mathcal D}_L(w)=\{\beta_2,\beta_3,\beta_4,\beta_5,\beta_7,\beta_8\}$. Let $I={\mathcal D}_L(w)$. 
Here 
\[w_0(I)=s_3 s_2 s_4 s_3 s_2 s_4 s_5 s_4 s_3 s_2 s_4 s_5 \cdot s_7 s_8 s_7 \text{\ \ and \ \ } w_0(I)w=s_1 s_6 s_7 s_8.\]
Since $w=w_0(I)c$ where $c=s_1 s_6 s_7 s_8$ is a standard Coxeter element, Theorem~\ref{conj:main} asserts
that $X_w$ is $L_I$-spherical in the complete flag variety for $E_8$.
\end{ex}

\begin{ex}[$F_4$ {cf.~\cite[Example~1.5]{Hodges.Yong1}}]
The $F_4$ diagram is \dynkin[label, edge length=0.5cm]F4. First suppose 
\[w=s_4 s_3 s_4 s_2 s_3 s_4 s_2 s_3 s_2 s_1 s_2 s_3 s_4 \text{\ ($I={\mathcal D}_L(w)=\{\beta_2,\beta_3,\beta_4\}$)}.\]
Then $w_0(I)= s_2 s_3 s_2 s_3 s_4 s_3 s_2 s_3 s_4$ and $w_0(I)w=s_1 s_2 s_3 s_4$ is standard
Coxeter. Hence $X_w$ is $L_I$-spherical. On the other hand if 
\[w'=s_2  s_1  s_4 s_3 s_2 s_1 s_3 s_2 s_4 s_3 s_2 s_1 \text{\  ($I={\mathcal D}_L(w')=\{\beta_2,\beta_4\}$)},\]
then $w_0(I)=s_2 s_4$ and $w_0(I)w=s_1 s_3 s_2 s_1 s_3 s_2 s_4 s_3 s_2 s_1$ is not standard
Coxeter and $X_w$ is not $L_{I}$-spherical.
\end{ex}

\begin{ex}[$D_4$]
The $D_4$ diagram is \dynkin[label, edge length=0.5cm]D4. Let 
\[w= s_3 s_2 s_3 s_4 s_2 s_1 s_2 \text{\ ($I={\mathcal D}_L(w)=\{\beta_2,\beta_3\}$).}\]
Thus $w_0(I)=s_2 s_3 s_2$ and $w_0(I)w=s_4 s_2 s_1 s_2$ is not standard Coxeter. Hence $X_w$
is not $L_I$-spherical. The interested reader can check $w$ is $I$-spherical in the (different) sense of \cite[Definition~1.2]{Hodges.Yong1}. Therefore, this $w$ provides a counterexample to \cite[Conjecture~1.9]{Hodges.Yong1} in
type $D_4$. This counterexample was also (implicitly) verified in \cite{GHY2} using a different method, namely 
Demazure character computations, the topic of Section~\ref{sec:4}.
\end{ex}

\section{An equivariant isomorphism}\label{sec:2}
The primary goal of this section is to construct a torus equivariant isomorphism from a specified affine subspace of the open cell of a Schubert variety
to the open cell of a distinguished Schubert subvariety. In what follows, we assume standard facts from the theory of algebraic groups. References we draw
upon are \cite{Humphreys, B91, LR08}.

Let $w \in W$. Let $n_w$ be a coset representative of $w$ in $N_G(T)$. By definition of $N_G(T)$ being the
normalizer of $T$ in $G$, $t \mapsto n_w t n_w^{-1}$ defines an automorphism $\gamma_w:T\to T$.

\begin{lemma}
\label{lemma:wellDefinedt_w}
The automorphism $\gamma_w$ does not depend on our choice of coset representative $n_w$.
\end{lemma}
\begin{proof}
Suppose that $m_w$ is another coset representative of $w$. Then $m_w = n_w s$ for some $s \in T$. Hence $m_w t m_w^{-1} = n_w s t s^{-1} n_w^{-1} = n_w t s s^{-1}n_w^{-1}=n_w t n_w^{-1}$.
\end{proof} 

In light of Lemma~\ref{lemma:wellDefinedt_w}, henceforth for $w \in W$ we will also let $w$ denote a coset representative of $w$ in $N_G(T)$. Let $X$ be a $T$-variety with action denoted by $\bigdot$. For each $w \in W$ we define an action $\bigdot_w$ on $X$ by $t \bigdot_w x = \gamma_w(t) \bigdot x$ for all $x \in X$ and $t \in T$. 

\begin{lemma}
\label{lemma:denseByAutomorphism}
For all $w \in W$, the $T$-variety $X$ has an open, dense $T$-orbit for the action $\bigdot$ if and only if it has an open, dense $T$-orbit for the action $\bigdot_w$. Indeed, the set of $T$-orbits in $X$ for these two actions is 
identical. 
\end{lemma}
\begin{proof}
Let $\mathcal{O}$ be a $T$-orbit in $X$ for the $\bigdot$ action. Let $x,y \in \mathcal{O}$ and $t \in T$ be such that $t \bigdot x = y$. As $\gamma_w$ is an automorphism, there exists a $t^{\prime} \in T$ such that $\gamma_w(t^{\prime}) = t$. Then 
\[t^{\prime} \bigdot_w x = \gamma_w (t^{\prime}) \bigdot x = t \bigdot x = y.\] 
Thus $\mathcal{O}$ is contained in the $T$-orbit $\mathcal{O}'$ of $x$ for the action $\bigdot_w$.  
The reverse containment is true by definition of $\bigdot_w$.
 The lemma
follows.
\end{proof}

For the remainder, we fix $\bigdot$ to be the restriction to $T$ of the action of $G$ on $G/B$ by left multiplication. 
The \emph{left inversion set} of $w \in W$ is 
\[{\mathcal I}(w):=\Phi^+ \cap w(\Phi^-) = \{ \alpha \in \Phi^+ | w^{-1}(\alpha) \in  \Phi^- \}.\] 
Recall two standard facts regarding left inversion sets \cite[Chapter 1]{Humphreys:grey}. For $w \in W$,
\begin{equation}
\label{eqn:leftinvIdent1}
|{\mathcal I}(w)| = \ell(w) = \dim_{\mathbb C} X_w,
\end{equation}
and
\begin{equation}
\label{eqn:leftinvIdent2}
{\mathcal I}(w_0(I)) = \Phi^+(I),
\end{equation}
where $\Phi(I)=\Phi({\mathfrak l}_I,T)$ is the root system for the adjoint action of $T$ on ${\mathfrak l}_I={\rm Lie}(L_I)$.

We say that an algebraic group $H$ is \emph{directly spanned} by its closed subgroups $H_1,\ldots,H_n$, in the given order, if the product morphism 
\[H_1 \times \cdots \times H_n \rightarrow H\] 
is bijective. For $w \in W$, define $U_w := U \cap w U^- w^{-1}$, where $U$ consists of the unipotent elements of $B$ and similarly, $U^{-}$ is
the unipotent part of $B^{-}:=w_0Bw_0$. This is a subgroup of $U$ that is closed and normalized by $T$. Hence, by \cite[\S 14.4]{B91}, $U_w$ is directly spanned, in any order, by the \emph{root subgroups} $U_{\alpha}$, $\alpha \in \Phi^+$, contained in $U_w$. Since by  \cite[Part II, 1.4(5)]{J03},
\begin{equation}\label{eqn:conjthat}
wU_{\alpha}w^{-1}=U_{w(\alpha)},
\end{equation}
  these are the $U_{\alpha}$ such that $\alpha \in \Phi^+ \cap w(\Phi^-) = {\mathcal I}(w)$. Thus
\begin{equation}
\label{eqn:leftinvIdent2.5}
U_w = \prod_{\alpha \in {\mathcal I}(w)} U_{\alpha},
\end{equation}
where the products $U_{\alpha}$ may be taken in any order.

\begin{lemma}\label{prop:anyorder}
For a coset $wB \in G/B$, we have 
\begin{equation}
\label{eqn:leftinvIdent3}
X_w^{\circ} := BwB = U_w wB = \prod_{\alpha \in {\mathcal I}(w)} U_{\alpha}\,\,wB.
\end{equation}
Moreover, $X_w^{\circ}$ is isomorphic to the affine space $\mathbb{A}^{\ell(w)}$ (as varieties).
\end{lemma}
\begin{proof}
It is a well-known fact that $U_w$ is isomorphic to $X_w^{\circ}$ (as varieties) under the map $u \mapsto uwB,\,u\in U_w$~\cite[\S 14.12]{B91}. The final equality in \eqref{eqn:leftinvIdent3} is \eqref{eqn:leftinvIdent2.5}. By~\cite[Remark in \S 14.4]{B91}, $U_w$ is isomorphic, as a variety, to $\mathbb{A}^{\ell(w)}$.
\end{proof}
 
We say that $w=uv\in W$ is \emph{length additive} if $\ell(uv)=\ell(u)+\ell(v)$. Under this
hypothesis, by \cite[Ch. VI, \S 1, Cor. 2 of Prop. 17]{B02} one has 
\[{\mathcal I}(uv)={\mathcal I}(u)\sqcup u({\mathcal I}(v)).\] 
Therefore, in particular, if we assume
$w_0(I)d \in W$ is \emph{length additive}, then 
\begin{equation}
\label{eqn:splittingUpw0Id}
{\mathcal I}(w_0(I)d)={\mathcal I}(w_0(I))\sqcup w_0(I)({\mathcal I}(d)).
\end{equation} 

Define   
\[V_d := w_0(I) U_d w_0(I)^{-1}=w_0(I)U_d w_0(I).\] 
\begin{lemma}\label{lemma:iscldsub}
$V_d$ is a closed subgroup of $U_{w_0(I)d}$ that is normalized by $T$.
\end{lemma}
\begin{proof}
Since $U_d$ is a closed subgroup normalized by $T$, so is $V_d$. Indeed, $V_d$ is a subgroup of $U_{w_0(I)d}$ since
\begin{equation}
\label{eqn:Vddef}
V_d = w_0(I) \prod_{\alpha \in {\mathcal I}(d)} U_{\alpha} w_0(I) = \prod_{\alpha \in w_0(I)({\mathcal I}(d))}  U_{\alpha} \leq U_{w_0(I)d}, 
\end{equation}
where the first equality is \eqref{eqn:leftinvIdent2.5}, the second is (\ref{eqn:conjthat}), and the subgroup
claim is \eqref{eqn:leftinvIdent2.5} and~\eqref{eqn:splittingUpw0Id}.
\end{proof}

\begin{lemma}\label{lemma:directlyspanned}
$U_{w_0(I)d}$ is directly spanned by $U_{w_0(I)}$ and $V_d$:
\begin{equation}
\label{eqn:carefulDirectSpanNotSemidirectProd}
U_{w_0(I)d} = U_{w_0(I)} V_d = V_d U_{w_0(I)}.
\end{equation} 
\end{lemma}
\begin{proof}
This follows from \eqref{eqn:leftinvIdent2.5}, \eqref{eqn:splittingUpw0Id}, and \eqref{eqn:Vddef} combined.
\end{proof}

Define 
\[ \tilde{O} := V_d w_0(I)dB \subseteq G/B. \] 
\begin{lemma}
$\tilde{O}$ is $T$-stable for the action $\bigdot$.
\end{lemma}
\begin{proof}
The claim follows since
\[V_d w_0(I)dB  = (t V_d t^{-1}) tw_0(I)dB \subseteq  V_d w_0(I)dB,\]
where the final step follows from the fact that $V_d$ is normalized by $T$ and that $w_0(I)dB$ is a $T$-fixed point in $G / B$. \end{proof}

The following is the main point of this section:

\begin{proposition}
\label{prop:TequivIso}
If $w_0(I)d \in W$ is length additive then \[ X_{w_0(I)d}^{\circ} = U_{w_0(I)d}\,w_0(I)dB. \]
Hence $\tilde{O} \subset X_{w_0(I)d}^{\circ}$. Moreover, $\tilde{O}$ with the $T$-action $\bigdot$ is $T$-equivariantly isomorphic to $X_{d}^{\circ}$ with the $T$-action $\bigdot_{w_0(I)}$.
\end{proposition}
\begin{proof}
By \eqref{eqn:leftinvIdent3}, $X_{w_0(I)d}^{\circ} = U_{w_0(I)d}\,w_0(I)dB$. Combining this with
Lemma~\ref{lemma:iscldsub}, one concludes that $\tilde{O} \subseteq X_{w_0(I)d}^{\circ}$.
Define a map 
\begin{alignat*}{3}
  \phi: {\tilde O} &\longrightarrow X_{d}^{\circ} \\
  aB &\longmapsto w_0(I)aB.
\end{alignat*}
Now, 
\[\phi(\tilde{O}) = w_0(I) V_d\,w_0(I)dB  
= U_d\,dB = X_{d}^{\circ},\] 
where the second equality is by the definition of $V_d$,  and the final equality is Lemma~\ref{prop:anyorder}. Thus $\phi$ is well-defined and surjective.

As $\phi$ is simply left multiplication by $w_0(I)$ it is injective. Further, by Lemma~\ref{prop:anyorder} $X_{d}^{\circ}$ is isomorphic as a variety to $\mathbb{A}^{\ell(d)}$, and thus is smooth, and in particular normal. Hence, by Zariski's main lemma, $\phi$ is an isomorphism of varieties. 

To see that $\phi$ is $T$-equivariant for the specified actions, let $t \in T$ and $aB \in \tilde{O}$. Then \[ \phi(t \bigdot aB) = w_0(I)taB = w_0(I)tw_0(I)w_0(I)aB = \gamma_{w_0(I)}(t) \bigdot \phi(aB) = t \bigdot_{w_0(I)} \phi(aB).\qedhere\]\end{proof}

\section{Proof of the main result}\label{sec:3}
We need a lemma examining the $L_I$-action on $\tilde{O}$. This lemma is then used in conjunction with Proposition~\ref{prop:TequivIso} to prove our main result.

Let $B_{L_I}=L_I\cap B$ and let $U_{L_I}$ be the unipotent radical of $B_{L_I}$. Then $B_{L_I}$ is a Borel subgroup in $L_I$~\cite[\S 14.17]{B91} with $U_{L_I}=B_{L_I}\cap U$ and $B_{L_I}=T \ltimes U_{L_I}$. Since $L_I$ is the subgroup of $G$ generated by $T$ and  $\{U_{\alpha} \,\mid\, \alpha \in \Phi(I)\}$~\cite[\S 3.2.2]{LR08}, it is straightforward to show that 
\[U_{L_I} = \prod_{\alpha \in \Phi^+(I)} U_{\alpha},\] 
where the product is taken in any order~\cite[\S 14.4]{B91}.

\begin{lemma}
\label{lemma:orbitDynamics}
Let $w = w_0(I) d \in W$ be length additive. Let $x \in X_{w_0(I)d}^{\circ} \setminus \tilde{O}$ and $y,z \in \tilde{O}$.
\begin{itemize}
\item[(i)] $uy \notin \tilde{O}$ for all $u \in U_{L_I}$ with $u \neq e$.
\item[(ii)] $tx \notin \tilde{O}$ for all $t \in T$.
\item[(iii)] There exists $b \in B_{L_I}$ such that $by=z$ if and only if there exists $t \in T$ such that $ty = z$.
\end{itemize}
\end{lemma}
\begin{proof}
(i) We have \[ U_{L_I} = \prod_{\alpha \in \Phi^+(I)} U_{\alpha} = U_{w_0(I)}, \] where the final equality is \eqref{eqn:leftinvIdent2.5}. Thus $u \in U_{w_0(I)}$.

Since $y \in \tilde{O}$, we have that $y=v w_0(I)dB$ for some $v \in V_d$. By
Lemma~\ref{lemma:directlyspanned}, $uv \in U_{w_0(I)d} \setminus V_d$ for $u \neq e$. Thus  $uv w_0(I)dB \in X_{w_0(I)d}^{\circ} \setminus \tilde{O}$ by Lemma~\ref{prop:anyorder}.

\noindent (ii) This follows immediately from the fact that $\tilde{O}$ is $T$-stable.

\noindent (iii) The Borel $B_{L_I} = T \ltimes U_{L_I}$, and thus for all $b \in B_{L_I}$ we may express $b=tu$ for unique $t\in T, u \in U_{L_I}$. If $u \neq e$, then $uy \notin \tilde{O}$ by (i) and so $by=tuy\notin \tilde{O}$ by (ii). Hence, if $by=z$, then $u = e$ and $b = t \in T$. The converse direction is immediate since $T \subseteq B_{L_I}$.
\end{proof}

We now have the necessary ingredients to complete the proof of our main result.

\noindent \emph{Proof of Theorem~\ref{conj:main}:} ($\Leftarrow$) Let $w$ be $I$-spherical. Then $w = w_0(I) c$ is length additive and $c$ is a standard Coxeter element. Our goal is to exhibit a $x\!\in\!\tilde{O}$ such that $\dim(B_{L_I}\!\bigdot\!x)\!=\!\dim X_{w_0(I)c}^{\circ}$.

The Schubert variety $X_{c}$ is a toric variety~\cite{Karrup}; it contains an open, dense $T$-orbit $\mathcal{O}$ for the $T$-action $\bigdot$. Since $X_{c}^{\circ}$ is an open, dense subset of $X_{c}$, $\mathcal{O} \cap X_{c}^{\circ}$ is open and dense in $X_{c}^{\circ}$; since $X_{c}^{\circ}$ is $T$-stable we have that $\mathcal{O} \cap X_{c}^{\circ}$ is a $T$-orbit in $X_{c}^{\circ}$ for the $T$-action $\bigdot$. Lemma~\ref{lemma:denseByAutomorphism} implies that $\mathcal{O} \cap X_{c}^{\circ}$ is an open, dense $T$-orbit for the $T$-action $\bigdot_{w_0(I)}$.

By Proposition~\ref{prop:TequivIso}, there is a $T$-equivariant isomorphism $\phi:
\tilde{O}\to X_{c}^{\circ}$. Let 
\[\Theta=\phi^{-1}({\mathcal O}\cap X_c^{\circ});\] 
this is an open, dense $T$-orbit in $\tilde{O}$ for the $T$-action $\bigdot$. Let $x \in \Theta$. By Lemma~\ref{lemma:orbitDynamics}(iii), the isotropy subgroup $(B_{L_I})_x$ is equal to the isotropy subgroup $T_x$. By \cite[Proposition 1.11]{B09}, for any variety $X$ equipped with the action of an algebraic group $H$, the orbit $H \cdot x$, $x \in X$, is a subvariety of $X$ of dimension $\dim H - \dim H_x$,
\begin{equation}
\label{eqn:orbitDim}
\dim(H \cdot x) = \dim H - \dim H_x.
\end{equation}
The above combine to imply that
\begin{equation}
\label{eqn:stabDim}
\dim(B_{L_I})_x = \dim T_x = \dim T - \dim(T \bigdot x) = \dim T - \dim \Theta = \dim T - \ell(c).
\end{equation}
We conclude, applying \eqref{eqn:orbitDim} and \eqref{eqn:stabDim}, that
\begin{align*}
\dim(B_{L_I} \cdot x) & = \dim B_{L_I} - \dim (B_{L_I})_x\\
&  = \ell(w_0(I)) + \dim T - (\dim T - \ell(c))\\
&  = \ell(w_0(I)) + \ell(c) \\
& = \ell(w_0(I)c),
\end{align*}
and thus there exists an dense $B_{L_I}$-orbit in $X_{w_0(I)c}$. Indeed, this dense orbit must also be open in its closure by~\cite[Proposition 1.8]{B91}. Hence, $X_{w_0(I)c}$ is $L_I$-spherical.

\noindent($\Rightarrow$) Suppose $w$ is not $I$-spherical. Then $w = w_0(I) d$ where $d$ is not a standard
Coxeter element. Moreover, by the hypothesis that $I\subseteq {\mathcal D}_L(w)$, this factorization 
is length additive. 

The Schubert variety $X_{d}$ is not a toric variety for the $\bigdot$ action of $T$~\cite{Karrup}. If $X_{d}^{\circ}$ contained an open, dense $T$-orbit, then $X_{d}$ would be a toric variety for $\bigdot$. Thus $X_{d}^{\circ}$ is not a toric variety for $\bigdot$. In general, a normal $G$-variety is spherical if and only if it has finitely many $B$-orbits
(see \cite[Theorem~2.1.2]{Perrin}). If $G=T$ then $B=T$ and hence there are infinitely many $T$-orbits in $X_{d}^{\circ}$ for $T$-action $\bigdot$; and for the $T$-action $\bigdot_{w_0(I)}$ by Lemma~\ref{lemma:denseByAutomorphism}. 

By Proposition~\ref{prop:TequivIso}, 
$\tilde{O}$
is $T$-equivariantly isomorphic as an affine variety to $X_{d}^{\circ}$. Thus, there are infinitely many $T$-orbits in $\tilde{O}$ for $T$-action $\bigdot$. Let $\mathcal{O}_1$ and $\mathcal{O}_2$ be two such orbits, and $x_1 \in \mathcal{O}_1, x_2 \in \mathcal{O}_2$. The fact that $x_1$ and $x_2$ reside in different orbits implies that there does not exist a $t \in T$ such that $tx_1 = x_2$. Thus Lemma~\ref{lemma:orbitDynamics}(iii) implies 
$B_{L_I} \bigdot x_1 \cap B_{L_I} \bigdot x_2 = \emptyset$. 
As these were an arbitrary pair among the infinite $T$-orbits, there must be infinitely many $B_{L_I}$ orbits in $X_{w_0(I)d}^{\circ}$ and hence in $X_{w_0(I)d}$. We conclude that $X_{w_0(I)d}$ is not $L_I$-spherical by
the same result \cite[Theorem~2.1.2]{Perrin} mentioned above.
 \qed

\section{Application to Demazure modules}\label{sec:4}

As an application of these results we give a sufficient condition for a Demazure module to be a multiplicity-free $L_I$-module; equivalently, a sufficient condition for a Demazure character to be multiplicity-free with respect to the basis of irreducible $L_I$-characters.

Let $\mathfrak{X}(T)$ denote the lattice of weights of $T$; our fixed Borel subgroup $B$ determines a subset of dominant integral weights $\mathfrak{X}(T)^+ \subset \mathfrak{X}(T)$.
The finite-dimensional irreducible $G$-representations are indexed by $\lambda \in \mathfrak{X}(T)^+$. Denoting the associated representation by $V_{\lambda}$, there is a class of $B$-submodules of $V_{\lambda}$, first introduced by Demazure \cite{Demazure}, that are indexed by $w \in W$. If $v_{\lambda}$ is a nonzero highest weight vector, then the \emph{Demazure module} $V_{\lambda}^{w}$ is the minimal $B$-submodule of $V_{\lambda}$ containing $wv_{\lambda}$.

There is a geometric construction of these Demazure modules. For $\lambda \in \mathfrak{X}(T)^+$, let $\mathfrak{L}_{\lambda}$ be the associated line bundle on $G / B$. For $w \in W$, we write $\mathfrak{L}_{\lambda}|_{X_w}$ for the restriction of $\mathfrak{L}_{\lambda}$ to the Schubert subvariety $X_w \subseteq G/ B$. Then the Demazure module $V_{\lambda}^{w}$ is isomorphic to the dual of the space of global sections of $\mathfrak{L}_{\lambda}|_{X_w}$, that is 
\[V_{\lambda}^{w} \cong H^0(X_w, \mathfrak{L}_{\lambda}|_{X_w})^*.\] 
This geometric perspective highlights the fact that $V_{\lambda}^{w}$ is not just a $B$-module, but is in fact also a $L_I$-module via the action induced on $H^0(X_w, \mathfrak{L}_{\lambda}|_{X_w})$ by the left multiplication action of $L_I$ on $X_w$.

As $L_I$ is a reductive group over characteristic zero, any $L_I$-module decomposes into a direct sum of irreducible $L_I$-modules. Let $\mathfrak{X}_{L_I}(T)^+$ be the set of dominant integral weights with respect to the choice of maximal torus and Borel subgroup $T \subseteq B_I \subseteq L_I$. For $\mu \in \mathfrak{X}_{L_I}(T)^+$, let $V_{L_I,\mu}$ be the associated irreducible $L_I$-module. If $M$ is a $L_I$-module and 
\[ \displaystyle M = \bigoplus_{\mu \in \mathfrak{X}_{L_I}(T)^+} V_{L_I,\mu}^{\oplus m_{L_I,\mu}} \]
is the decomposition into irreducible $L_I$-modules, then we say that $M$ is a \emph{multiplicity-free $L_I$-module} if $m_{L_I,\mu} \in \{ 0, 1 \}$. Similarly, if $\mathrm{char}(M)$ is the formal $T$-character of $M$ and \[ \displaystyle \mathrm{char}(M) = \sum_{\mu \in \mathfrak{X}_{L_I}(T)^+} m_{L_I,\mu} \mathrm{char}(V_{L_I,\mu}), \]
then we say that $\mathrm{char}(M)$ is \emph{$I$-multiplicity-free} if $m_{L_I,\mu} \in \{ 0, 1 \}$.

The following argument was given for type $A$ in \cite[Theorem~4.13(II)]{Hodges.Yong1}. We include the general
type argument (which is essentially the same) for sake of completeness:
\begin{theorem}
\label{theorem:demazureConsequences}
Let $w \in W$ with $I \subseteq D_L(w)$. Then $X_w$ is $L_I$-spherical if and only if for all $\lambda \in \mathfrak{X}(T)^+$, the Demazure module $V_{\lambda}^{w}$ is multiplicity-free $L_I$-module.
\end{theorem}
\begin{proof}
Let $X$ be a quasi-projective, normal variety with the action of a complex, connected, reductive algebraic group $G$. Then $X$ is $G$-spherical if and only if the $G$-module $H^0(X,\mathfrak{L})$ is a multiplicity free $G$-module for all $G$-linearized line bundles $\mathfrak{L}$~\cite[Theorem 2.1.2]{Perrin}. 

All Schubert varieties $X_w \subseteq G / B$ are normal, quasi-projective varieties~\cite{Jos.Normal}. Further, as $L_I$ is reductive and we are in characteristic zero, $V_{\lambda}^{w}$ is a multiplicity-free $L_I$-module if and only if its dual space $(V_{\lambda}^{w})^* = H^0(X_w, \mathfrak{L}_{\lambda}|_{X_w})$ is a multiplicity-free $L_I$-module~\cite[\S 31.6]{Humphreys}. The above combines to imply our desired result once we show that the $L_I$-linearized line bundles on $X_w$ are precisely of the form $\mathfrak{L}_{\lambda}|_{X_w}$ for $\lambda \in \mathfrak{X}(T)^+$.

The line bundles, with non-trivial spaces of global sections, on $G/B$ are precisely $\mathfrak{L}_{\lambda}$, for $\lambda \in \mathfrak{X}(T)^+$;
these line bundles are all $G$-linearized~\cite[\S 1.4]{Brion.Lectures}. Every line bundle on $X_w$ is the restriction of a line bundle on $G/B$~\cite[Proposition 2.2.8]{Brion.Lectures}. We are done since the restriction $\mathfrak{L}_{\lambda}|_{X_w}$ of the $G$-linearized line bundle $\mathfrak{L}_{\lambda}$, for $\lambda \in \mathfrak{X}(T)^+$, is $L_I$-linearized. 
\end{proof}

\begin{corollary}
Let $w \in W$ be $I$-spherical for $I \subseteq D_L(w)$. For all $\lambda \in \mathfrak{X}(T)^+$, the Demazure module $V_{\lambda}^{w}$ is a multiplicity-free $L_I$-module.
\end{corollary}
\begin{proof}
By Theorem~\ref{conj:main}, if $w$ is $I$-spherical then $X_w$ is $L_I$-spherical. Therefore, by Theorem~\ref{theorem:demazureConsequences}, $V_{\lambda}^{w}$ is a multiplicity-free $L_I$-module for $\lambda \in \mathfrak{X}(T)^+$.
\end{proof}

\begin{corollary}
Let $w \in W$ be $I$-spherical for $I \subseteq D_L(w)$. For all $\lambda \in \mathfrak{X}(T)^+$, the Demazure character $\mathrm{char}(V_{\lambda}^{w})$ is $I$-multiplicity-free.
\end{corollary}

These two corollaries appear non-trivial from a combinatorial perspective, even for a \emph{specific choice} of dominant weight $\lambda$ with fixed $w\in W$. The Demazure character can be recursively computed using Demazure operators. There is also a combinatorial rule for the character in terms of crystal bases (in instantiations such as the \emph{Littelmann path model} or the \emph{alcove walk model}); see, e.g., the textbook \cite{Bump}. However, an argument based on these methods eludes in general type, although we have an argument in type $A$ \cite{GHY}.

\section*{Acknowledgements} 
RH was partially supported by an AMS-Simons Travel Grant. AY is partially supported by a Simons Collaboration
Grant and an NSF RTG in Combinatorics.

\end{document}